\documentclass[11pt, letterpaper,reqno]{amsart}

\usepackage{graphicx}
\usepackage[linktocpage=true,colorlinks,citecolor=blue,linkcolor=blue,urlcolor=black]{hyperref}
\usepackage{fullpage}
\usepackage{amssymb}
\usepackage{cite}
\usepackage{amsmath}
\usepackage{latexsym}
\usepackage{amscd}
\usepackage{tikz}
\usepackage{amsthm}
\usepackage{mathrsfs}
\usepackage{url}
\usepackage[utf8]{inputenc}
\usepackage[english]{babel}
\usepackage{amsfonts}
\usepackage{todonotes}
\usepackage{mathtools}
\usepackage{verbatim}
\usepackage[justification=centering]{caption}
\usepackage{enumerate}

\numberwithin{equation}{section}
\def\R{\mathbb R}
\def\Z{\mathbb Z}

\def\mc{\mathcal C}

\def\E{\mathbb E}

\theoremstyle{plain}
\newtheorem{theorem}{Theorem}

\newtheorem{lemma}[theorem]{Lemma}

\newtheorem{prop}[theorem]{Proposition}
\newtheorem*{theorem*}{Theorem}

\newtheorem{corollary}[theorem]{Corollary}
\newtheorem{conjecture}[theorem]{Conjecture}

\theoremstyle{remark}

\theoremstyle{definition}

\theoremstyle{remark}

\numberwithin{equation}{section}

\begin{document}
\title{Monochromatic components with many edges}

\author{David Conlon}
\address{Department of Mathematics, California Institute of Technology, Pasadena, CA 91125, USA}
\email{dconlon@caltech.edu}
\thanks{Conlon was supported by NSF Award DMS-2054452, Luo by NSF GRFP Grant DGE-1656518, and Tyomkyn by ERC Synergy Grant DYNASNET 810115 and GA\v{C}R Grant 22-19073S}

\author{Sammy Luo}
\address{Department of Mathematics, Stanford University, Stanford, CA 94305, USA}
\email{sammyluo@stanford.edu}
\thanks{}

\author{Mykhaylo Tyomkyn}
\address{Department of Applied Mathematics, Faculty of Mathematics and Physics, Charles University, 11800 Prague, Czech Republic}
\email{tyomkyn@kam.mff.cuni.cz}
\thanks{}

	\maketitle

\begin{abstract}
Given an $r$-edge-coloring of the complete graph $K_n$, what is the largest number of edges in a monochromatic connected component? This natural question has only recently received the attention it deserves, with work by two disjoint subsets of the authors resolving it for the first two special cases, when $r = 2$ or $3$. Here we introduce a general framework for studying this problem and apply it to fully resolve the $r = 4$ case, showing that any $4$-edge-coloring of $K_n$ contains a monochromatic component with at least $\frac{1}{12}\binom{n}{2}$ edges, where the constant $\frac{1}{12}$ is optimal only when the coloring matches a certain construction of Gyárfás.
\end{abstract}

\section{Introduction}

Given an $r$-coloring of the edges of the complete graph $K_n$, how large is the largest monochromatic connected component? A partial answer to this question was provided in 1977 by Gy\'arf\'as~\cite{gyarfas1977partition}, who showed that any such $r$-coloring always contains a monochromatic connected component with at least $n/(r-1)$ vertices. Moreover, this estimate is best possible whenever $r-1$ is a prime power and $n$ is a multiple of $(r-1)^2$. An alternative proof of this result, as a simple corollary of his fractional version of Ryser’s conjecture, was later found by F\"uredi~\cite{furedi81ryser, furedi89covering}, who also showed that if there is no affine plane of order $r-1$, then the bound can be improved to $(r-1)n/(r^2 - 2r)$. 

There are many variants of this question. For instance, what happens when the complete graph $K_n$ is replaced by another graph, say a subgraph of the complete graph~\cite{gyarfas2017mindeg} or the complete bipartite graph~\cite{debiasio2020bipartite} of high minimum degree? Or what happens when we insist that our component has small diameter~\cite{Ruszinko2012diameter}? 
Here we will be concerned with another variant, a rather basic one which has received surprisingly little attention in the literature, namely, given an $r$-coloring of $K_n$, what is the largest number of \emph{edges} in a monochromatic connected component?

This question was first raised by Conlon and Tyomkyn~\cite{conlon2021ramsey} because of its close relation with another problem, that of determining the Ramsey number for trails and circuits. However, the components problem is arguably the more fundamental question. If we write $M(n, r)$ for the largest natural number such that every $r$-coloring of $K_n$ contains a monochromatic connected component with at least $M(n, r)$ edges, then the main result of~\cite{conlon2021ramsey} may be interpreted as saying that $M(n, 2) = \frac{2}{9} n^2 + o(n^2)$. In fact, a more careful analysis of their argument implies that $M(n, 2) \geq \frac{1}{9} (2n^2 - n - 1)$, with, where divisibility allows it, the example consisting of two disjoint red cliques of orders $\frac{2n+1}{3}$ and  $\frac{n-1}{3}$ with all blue edges between showing that this is best possible.

To say something about the general case, we first look at Gy\'arf\'as' construction of $r$-colorings where each monochromatic component has at most $n/(r-1)$ vertices. As noted earlier, his construction, which relies on the existence of the affine plane of order $r-1$, works when $r-1$ is a prime power and $n$ is a multiple of $(r-1)^2$. Concretely, the affine plane of order $r-1$ corresponds to a copy of $K_{(r-1)^2}$ together with $r$ different decompositions of this graph into $r-1$ vertex-disjoint copies of $K_{r-1}$ (that is, $r$ different $K_{r-1}$-factors) with the property that any edge is contained in exactly one of the $r(r-1)$ copies of $K_{r-1}$. By giving the edges in the $i$th $K_{r-1}$-factor color $i$, we obtain an $r$-coloring where every monochromatic component has at most $r-1$ vertices. Moreover, when $n$ is a multiple of $(r-1)^2$, we can simply blow up this coloring 
to obtain an $r$-coloring where every monochromatic component has at most $n/(r-1)$ vertices.

As noted by Conlon and Tyomkyn~\cite{conlon2021ramsey}, essentially the same construction works in the edge case to show that, when $r-1$ is a prime power, there are $r$-colorings where each monochromatic component has at most $(\frac{1}{r(r-1)} + o(1))\binom{n}{2}$ edges (the only caveat is that we should use each color roughly the same number of times within each of the blown-up vertices).
They also showed that this bound is correct up to a constant and conjectured that, for $r = 3$, it is asymptotically tight. That this is the case was verified by Luo~\cite{luo2021connected}, who proved that $M(n, 3) = \lceil \frac{1}{6} \binom{n}{2}\rceil$ for $n$ sufficiently large. Moreover, by giving a tight lower bound for the largest number of edges in a connected component in a graph of given density, he was able to show that $M(n, r) \geq \frac{1}{r^2} \binom{n}{2}$ in the general case, a result which was later strengthened to  $M(n, r) \geq \frac{1}{r^2-r+\frac{5}{4}} \binom{n}{2}$ in a revised version of the paper. Our concern here will be with the following conjectural improvement to this bound.

\begin{conjecture} \label{conj:main}
For any natural numbers $n$ and $r$ with $r \geq 3$, 
\[M(n, r) \geq \left\lceil \frac{1}{r(r-1)} \binom{n}{2} \right\rceil.\]
Moreover, when there is no affine plane of order $r-1$, there exists a constant $\varepsilon_r > 0$ such that
\[M(n,r) \geq \left( \frac{1}{r(r-1)} + \varepsilon_r \right) \binom{n}{2}.\]
\end{conjecture}


The result of \cite{luo2021connected} proves this conjecture when $r = 3$, while the result of~\cite{conlon2021ramsey} shows that the conjectured bound does not extend to the case $r=2$.
Our main result is a proof of the next open case, when $r = 4$. Note that, in this case, Gy\'arf\'as' construction corresponds to a $4$-coloring of $K_9$ where each color class is the union of three vertex-disjoint triangles. In the statement below, by saying that a coloring matches Gyárfás' construction, we mean that the set of components and the intersection pattern of their vertex sets match those in this construction.

\begin{theorem} \label{thm:4main}
In every $4$-coloring of the edges of $K_n$, there is a monochromatic component with at least $\frac{1}{12} \binom{n}{2}$ edges. That is, $M(n, 4) \geq \lceil \frac{1}{12} \binom{n}{2} \rceil$. Moreover, unless the coloring matches Gyárfás's construction, there is a monochromatic component with at least $\left(\frac{1}{12}+\varepsilon\right)\binom{n}{2}$ edges, where $\varepsilon=\frac{2}{14+\sqrt{96}}-\frac{1}{12}>0.0007$.
\end{theorem}

Our proof of Theorem~\ref{thm:4main} consists of first showing that any $4$-coloring of $K_n$ has one of a bounded number of component structures and then that each such component structure contains a component with enough edges. For instance, one of the possible component structures is that each color has precisely three components. But then one of these $12$ components clearly contains at least $1/12$ of the edges, as required. For the other possible component structures, our arguments are not usually so simple, relying instead on a key observation, Proposition~\ref{prop:mainbound} below. 
This says that if a certain union of components is large in the vertex sense, but none of these components is large in the edge sense, then some one of the remaining components will be large in the edge sense. In fact, even this is not quite enough and, inspired by F\"uredi's approach to the vertex case, we must allow for weighted or fractional unions of components. We will describe our general framework and how it may be applied in more detail in the next section.

\section{A general framework}

Suppose $r\geq 2$ and fix an $r$-coloring of the edges of the complete graph $G=(V,E)\mathrel{\widehat{=}} K_n$. For $1\leq i\leq r$, let $G_i$ be the ``graph of color $i$'', i.e., the subgraph of $K_n$ formed by the edges with color $i$, where we include vertices that are isolated in that color in the graph. In particular, $|G_i|=n$ for all $i$. Let $\mathcal{C}_i$ be the set of connected components in $G_i$ and let $\mathcal{C}$ be the set of all monochromatic components, i.e., $\mathcal{C}=\bigcup_{1\leq i\leq r}\mathcal{C}_i$. Since each vertex of $G$ is in exactly one component of each color, we have $$\sum_{C\in \mathcal{C}}|V(C)|=r|V(G)|=rn.$$ Similarly, since each edge of $G$ is in exactly one component of exactly one color, we have $$\sum_{C\in \mathcal{C}}|E(C)|=|E(G)|= \binom{n}{2}.$$ The following observation will be central to our approach.

\begin{prop}
\label{prop:mainbound}
Let $\mathcal{X}\subseteq \mathcal{C}$ be a set of connected monochromatic components in an $r$-coloring of $K_n$ and let $x=|\mathcal{X}|$. Suppose $\gamma\in [0,r]$ and $z\in \R^+$ are constants such that $\sum_{C\in \mathcal{X}} |V(C)|\geq \gamma n$ and $\max_{C\in \mathcal{C}}|E(C)| \leq z \binom{n}{2}$. Then
\begin{equation}\label{eqn:fracbound} 
    z(r-\gamma)^2\geq \max(1-xz,0)^2.
\end{equation}
\end{prop}

\begin{proof}
If $\gamma=r$, then $\sum_{C\in \mathcal{X}} |V(C)|\geq rn = \sum_{C\in \mathcal{C}} |V(C)|$, so $\mathcal{X}=\mathcal{C}$. Then
\[
xz=\sum_{C\in\mathcal{C}} z \geq \sum_{C\in\mathcal{C}}\frac{|E(C)|}{\binom{n}{2}}=1, 
\]
so both sides of \eqref{eqn:fracbound} are zero. Thus, we can assume $\gamma < r$.

For $1\leq i\leq r$, let $H_i$ be the induced subgraph of $G_i$ on $V(G)\setminus \bigcup_{C\in \mathcal{C}_i\cap \mathcal{X}} V(C)$.  
Then
\[
\sum_{i=1}^r |E(H_i)| = |E(K_n)|-\sum_{C\in \mathcal{X}}|E(C)|\geq \max\left((1-xz) \binom{n}{2},0\right),
\]
while
\[
\sum_{i=1}^r |V(H_i)| = rn-\sum_{C\in \mathcal{X}}|V(C)|\leq (r-\gamma)n.
\]
Note that $|V(H_i)|=0$ implies $|E(H_i)|=0$. Thus, by a standard averaging argument, there is some $j$ with $|V(H_j)|\neq 0$ and
\[
\frac{|E(H_j)|}{|V(H_j)|}\geq \frac{\sum_{i=1}^r |E(H_i)|}{\sum_{i=1}^r |V(H_i)|}\geq \frac{n-1}{2}\max\left(\frac{1-xz}{r-\gamma},0\right).
\]
Let $\mc_j'$ be the set of components of $H_j$, so, by the same averaging argument, there is some $C'\in \mc_j'$ such that
\begin{equation}\label{eq:largecomp}
\frac{|E(C')|}{|V(C')|}\geq \frac{\sum_{C\in \mc_j'}|E(C)|}{\sum_{C\in \mc_j'}|V(C)|}=\frac{|E(H_j)|}{|V(H_j)|}\geq \frac{n-1}{2}\max\left(\frac{1-xz}{r-\gamma},0\right).
\end{equation}
Since $|E(C')|\leq \binom{|V(C')|}{2}$, we have $$\frac{n-1}{2}\geq \frac{|V(C')|-1}{2}\geq \frac{|E(C')|}{|V(C')|},$$
which, combined with~\eqref{eq:largecomp}, implies $\frac{1-xz}{r-\gamma}\leq 1$. We therefore have
\begin{align*}
    z\binom{n}{2}\geq |E(C')|&= \frac{|E(C')|}{|V(C')|} |V(C')|\geq \frac{|E(C')|}{|V(C')|}\left(\frac{2|E(C')|}{|V(C')|}+1\right)  \\
    &=\binom{\frac{2|E(C')|}{|V(C')|}+1}{2}\geq \binom{(n-1)\max( \frac{1-xz}{r-\gamma},0)+1}{2}\\
    &\geq \frac{\max(1-xz,0)^2}{(r-\gamma)^2}\binom{n}{2},
\end{align*}
which rearranges to the desired bound.
\end{proof}

Note that taking $\mathcal{X}=\emptyset$, we have $x=\gamma=0$, so that $z\geq \frac{1}{r^2}$, which immediately yields the simple lower bound $M(n,r)\geq \frac{1}{r^2}\binom{n}{2}$, 
as proved in Corollary~4 of~\cite{luo2021connected}. 
For $1-xz\geq 0$, $x>0$, and $0\leq z\leq 1$, \eqref{eqn:fracbound} is equivalent to
\begin{equation}
\label{eqn:zbound}
z\geq \frac{(r-\gamma)^2 + 2x -\sqrt{((r-\gamma)^2+2x)^2-4x^2}}{2x^2}=\frac{2}{(r-\gamma)^2+ 2x +\sqrt{((r-\gamma)^2+2x)^2-4x^2}}.
\end{equation}
In order to go beyond the bound $M(n,r)\geq \frac{1}{r^2-r+\frac{5}{4}}\binom{n}{2}$ from Theorem~1 of~\cite{luo2021connected}, we must investigate the possible component structures in our $r$-coloring. For instance, in the case $r=3$, the proof of Theorem~2 in~\cite{luo2021connected} shows that there are three possible structures for the components in our coloring. Either: 
\begin{enumerate}[(a)]
    \item Some color has exactly one component ($\gamma=1,x=1$ above),
    \item Each color has exactly two components ($\gamma=3, x=6$), or
    \item There is a component of each color such that every vertex is covered by at least two of these three components ($\gamma=2, x=3$).
\end{enumerate}
Applying Proposition~\ref{prop:mainbound} to cases~(a) and~(c) yields, in each case, a lower bound on $z$ higher than the tight bound of $\frac{1}{6}$, while case~(b), with $r-\gamma=0$, implies $1-xz\leq 0$, so $z\geq \frac{1}{x}=\frac{1}{6}$, which is tight.

\vspace{5mm}

For general $r$, in order to prove a lower bound of the form $M(n,r)\geq z\binom{n}{2}$, it suffices to find, in any given $r$-coloring, a set $\mathcal{X}\subseteq \mathcal{C}$ yielding values of $x\in \Z^+$ and $\gamma\in [0,r]$ such that $x\leq \frac{1}{z}$ and \eqref{eqn:fracbound} does not hold, that is, $z (r-\gamma)^2 \leq (1-xz)^2$.
This rearranges to
\begin{equation}
    \label{eqn:xbound}
    x\leq \frac{1}{z}-\frac{r-\gamma}{\sqrt{z}}.
\end{equation}

We can state our conclusions concisely as follows.

\begin{corollary}
\label{cor:intbound}
Let $z\in \R^+$. If there exists a set of components $\mathcal{X}\subseteq \mathcal{C}$ such that, for some $x\in \Z^+$ and $\gamma\in [0,r]$, $|\mathcal{X}|=x$, $\sum_{C\in \mathcal{X}} |V(C)|\geq \gamma n$, and \eqref{eqn:xbound} holds, then there is a component $C\in \mathcal{C}$ with at least $z \binom{n}{2}$ edges.
\end{corollary}

In fact, a simple probabilistic argument allows us to strengthen Corollary~\ref{cor:intbound} to a fractional form allowing for non-integer values of $x$.

\begin{prop}
\label{prop:fracbound}
Let $z\in \R^+$. If there exists a function $w:\mathcal{C}\to[0,1]$ such that, for some $x\in \R^+$ and $\gamma\in [0,r]$, $\sum_{C\in \mathcal{C}}w(C)=x$, $\sum_{C\in \mathcal{C}}w(C)|V(C)|\geq \gamma n$, and \eqref{eqn:xbound} holds, then there is a component $C\in \mathcal{C}$ with at least $z \binom{n}{2}$ edges.
\end{prop}

\begin{proof}
Our proof takes advantage of the fact that \eqref{eqn:xbound} is linear in $x$ and $\gamma$. Given a function $w$ with the required properties, we construct a random subset $\mathcal{X}\subseteq \mathcal{C}$ by taking each component $C\in \mathcal{C}$ with probability $w(C)$. Abusing notation by letting $x$ and $\gamma$ be the random variables with $x=|\mathcal{X}|$ and $\gamma=\frac{1}{n}\sum_{C\in \mathcal{X}} |V(C)|$, we have
\begin{align*}
    \E\left(x+\frac{r-\gamma}{\sqrt{z}}\right)&=\sum_{C\in \mathcal{C}}w(C)+\frac{r-\frac{1}{n}\sum_{C\in \mathcal{C}}w(C)|V(C)|}{\sqrt{z}}\\
    &\leq x+\frac{r-\gamma}{\sqrt{z}}\leq \frac{1}{z},
\end{align*}
so there is some choice of $\mathcal{X}$ that satisfies the conditions of Corollary~\ref{cor:intbound}, yielding a component with the desired size.
\end{proof}

In practice, we will apply Proposition~\ref{prop:fracbound} in the following form, allowing us to work with the convenient bound \eqref{eqn:zbound}.

\begin{corollary}
\label{cor:fracbound}
Let $w:\mathcal{C}\to[0,1]$, $x\in \R^+$, and $\gamma\in [0,r]$ and suppose $\sum_{C\in \mathcal{C}}w(C)=x$ and $\sum_{C\in \mathcal{C}}w(C)|V(C)|\geq \gamma n$. Then \eqref{eqn:zbound} holds for $z=\frac{1}{\binom{n}{2}}\max_{C\in \mathcal{C}}|E(C)|$.
\end{corollary}

Note that the condition $\sum_{C\in \mathcal{C}}w(C)|V(C)|\geq \gamma n$ is satisfied if, for every vertex $v\in V(G)$, we have
\[
\sum_{C\ni v} w(C)\geq \gamma,
\]
i.e., the function $w$ gives a fractional cover of the vertices by components that cover each vertex with weight at least $\gamma$. This allows us to convert our problem into a linear program, akin to F\"uredi's approach to the vertex case, namely, we wish to minimize $x=\sum_{C\in\mathcal{C}} w(C)$ subject to the constraints that $0 \leq w(C) \leq 1$ for all $C$ and $\sum_{C\ni v} w(C)\geq \gamma$
for all $v\in V(G)$. In particular, to show that $z=\frac{1}{r(r-1)}$ works, it would suffice to show that $x\leq r(r-1)-(r-\gamma)\sqrt{r(r-1)}$.

\section{The case $r=4$}

In the case $r=4$, we can investigate the set of possible component structures as in the case $r=3$, though the analysis is now considerably more intricate. Recall that our aim is to show there exists a component with at least $\frac{1}{12}\binom{n}{2}$ edges; moreover, we would like to show that this bound is only asymptotically tight when the components match the extremal configuration described by Gyárfás in \cite{Gyarfas2011Survey}. 

We will make use of the following fact shown in the course of handling the $r=3$ case in~\cite{luo2021connected}. An equivalent result also appears as Lemma~4.19 in \cite{debiasio2021ryser}.

\begin{lemma}
\label{lem:2colorbip}
In any $2$-coloring of the complete bipartite graph between two vertex sets $A_1$ and $A_2$, one of the following holds:
\begin{enumerate}[(a)]
    \item Some color has exactly one component,
    \item Each color has exactly two components, each of which intersects both $A_1$ and $A_2$, or
    \item There is one component of each color such that the intersection of their vertex sets contains one of $A_1$ and $A_2$ and their union contains both.
\end{enumerate}
In each of these cases, one may assign weights of $\frac{1}{2}$ or $1$ to the components involved to get weights summing to at most $2$ that cover every vertex involved to weight at least $1$.
\end{lemma}
\begin{figure}[h]
    \centering
    \includegraphics[scale=0.6]{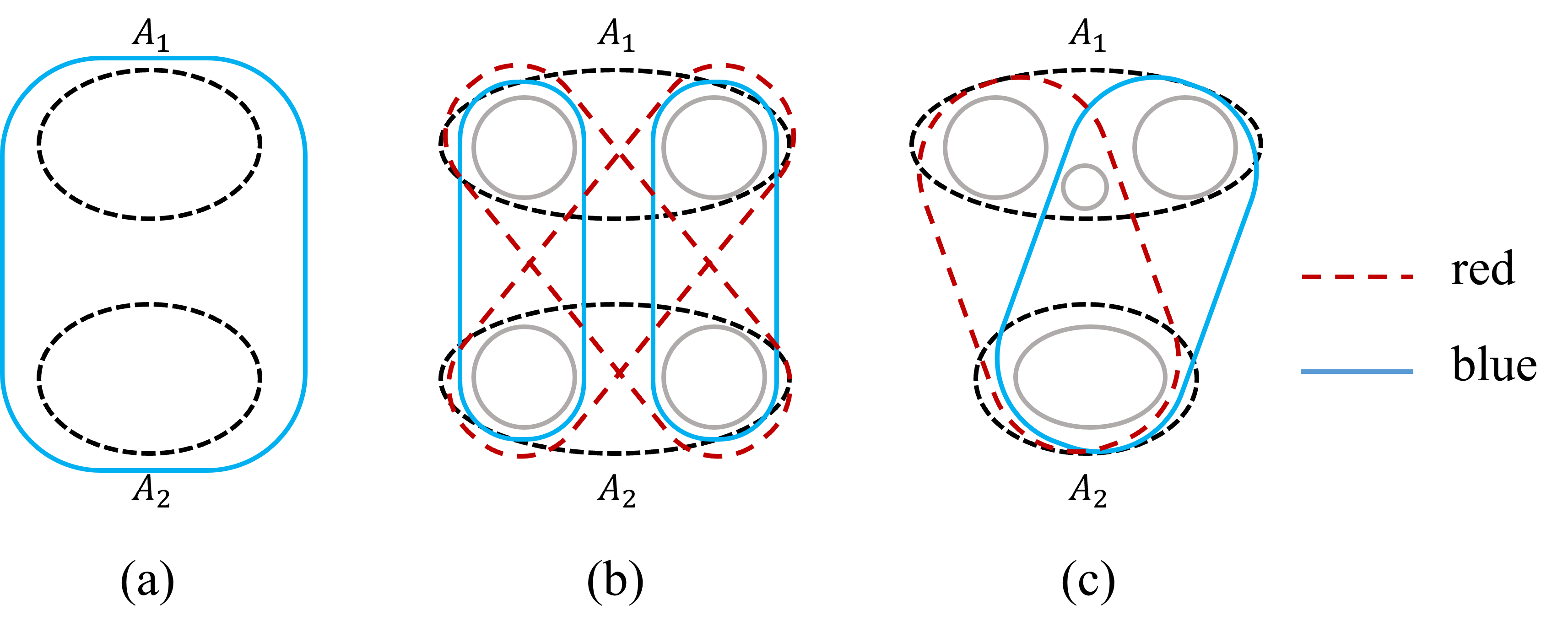}
    \caption{The cases of Lemma~\ref{lem:2colorbip}}
    \label{fig:2colorbip}
\end{figure}

It will sometimes be convenient to specify further subcases within the cases of Lemma~\ref{lem:2colorbip}. We say that a pair of vertex sets $(A_1,A_2)$ ``satisfies case~(a) for color $c$'' if $c$ is a color with exactly one component covering $A_1\cup A_2$. 
We also say that $(A_1,A_2)$ ``satisfies case~(c) directed toward $A_1$'' (or, equivalently, directed away from $A_2$) if $A_1$ is contained within the intersection of the two components. 

Define the $G_i$ and $\mathcal{C}_i$ as before and, for convenience, name the colors red, orange, yellow, and blue. We begin by establishing a lower bound on $\max_{C\in \mc} |E(C)|$ in the following ``degenerate'' case where some three components with distinct colors cover the entire vertex set. 

\begin{lemma}
\label{lem:case3comps}
If, in a $4$-coloring of the edges of $K_n$, there are three components $C_1,C_2,C_3\in \mc$ of distinct colors such that $C_1\cup C_2\cup C_3=V$, then
\[
\frac{\max_{C\in \mc} |E(C)|}{\binom{n}{2}} \geq \frac{2}{14+\sqrt{96}} > \frac{1}{12}.
\]
\end{lemma}

\begin{proof}
If there is one component $C_1$ with $V(C_1)=V$, then applying Proposition~\ref{prop:mainbound} with $\mathcal{X}=\{C_1\}$ and $(\gamma,x)=(1,1)$ yields, via \eqref{eqn:zbound}, the bound
\[
z\geq \frac{2}{11+\sqrt{117}} > \frac{2}{22} > \frac{2}{14+\sqrt{96}},
\]
as desired.

Next, suppose that there are two components $C_1,C_2$ with $C_1\cup C_2=V$, but with both $V_1 \coloneqq C_1\setminus C_2$ and $V_2 \coloneqq C_2\setminus C_1$ non-empty. Since we have not yet distinguished the colors in any way, we can assume without loss of generality that $C_1$ is red and $C_2$ is orange. Then all edges between $V_1$ and $V_2$ are either yellow or blue. Applying Lemma~\ref{lem:2colorbip} to the pair of vertex sets $(V_1,V_2)$ yields a way to choose weights on the yellow and blue components summing to at most $2$ such that every vertex in $V_1\cup V_2$ is covered by components with weights summing to at least $1$. Starting from this choice of weights and adding a weight of $1$ to each of $C_1$ and $C_2$ then allows us to apply Corollary~\ref{cor:fracbound} with $(\gamma,x)=(2,4)$, which yields the bound 
\[
z\geq \frac{2}{12+\sqrt{80}} > \frac{2}{21} > \frac{2}{14+\sqrt{96}},
\]
which again suffices.

The remaining case to consider is where there are three components $C_1,C_2,C_3$ with $C_1\cup C_2\cup C_3=V$, but all of $V_1\coloneqq C_1\setminus (C_2\cup C_3)$, $V_2\coloneqq C_2\setminus (C_3\cup C_1)$, and $V_3\coloneqq C_3\setminus (C_1\cup C_2)$ are non-empty. Without loss of generality, we can assume that $C_1$ is red, $C_2$ is orange, and $C_3$ is yellow. For each pair of the $V_i$, only two colors are possible on the edges between them, one of which is blue. Our aim is to apply Lemma~\ref{lem:2colorbip} to each of these pairs and then combine the results into an appropriate choice of weights on the components of the overall coloring of $K_n$.

\begin{figure}[h]
    \centering
    \includegraphics[scale=0.6]{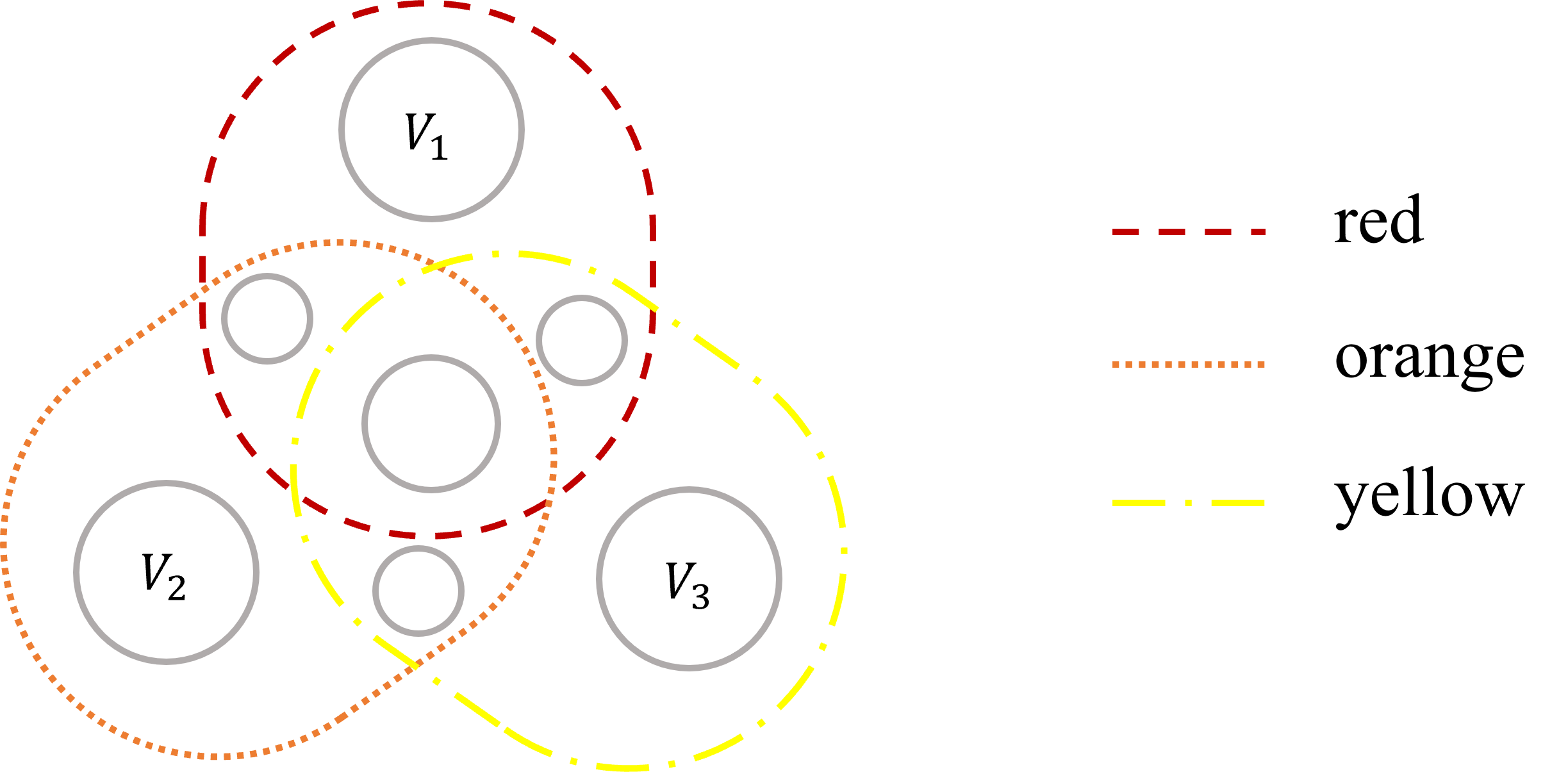}
    \caption{The vertex partition in Lemma~\ref{lem:case3comps}}
    \label{fig:3complemma}
\end{figure}

For each of the pairs $(V_i,V_j)$, we apply Lemma~\ref{lem:2colorbip} to the complete bipartite graph between $V_i$ and $V_j$, yielding one of the cases (a), (b), or (c) described therein. We further split case (a) into (a1) and (a2), depending on whether the corresponding component is blue (case (a1)) or not (some cases are not mutually exclusive, but this will not be an issue).

\begin{figure}[h]
    \centering
    \includegraphics[scale=0.55]{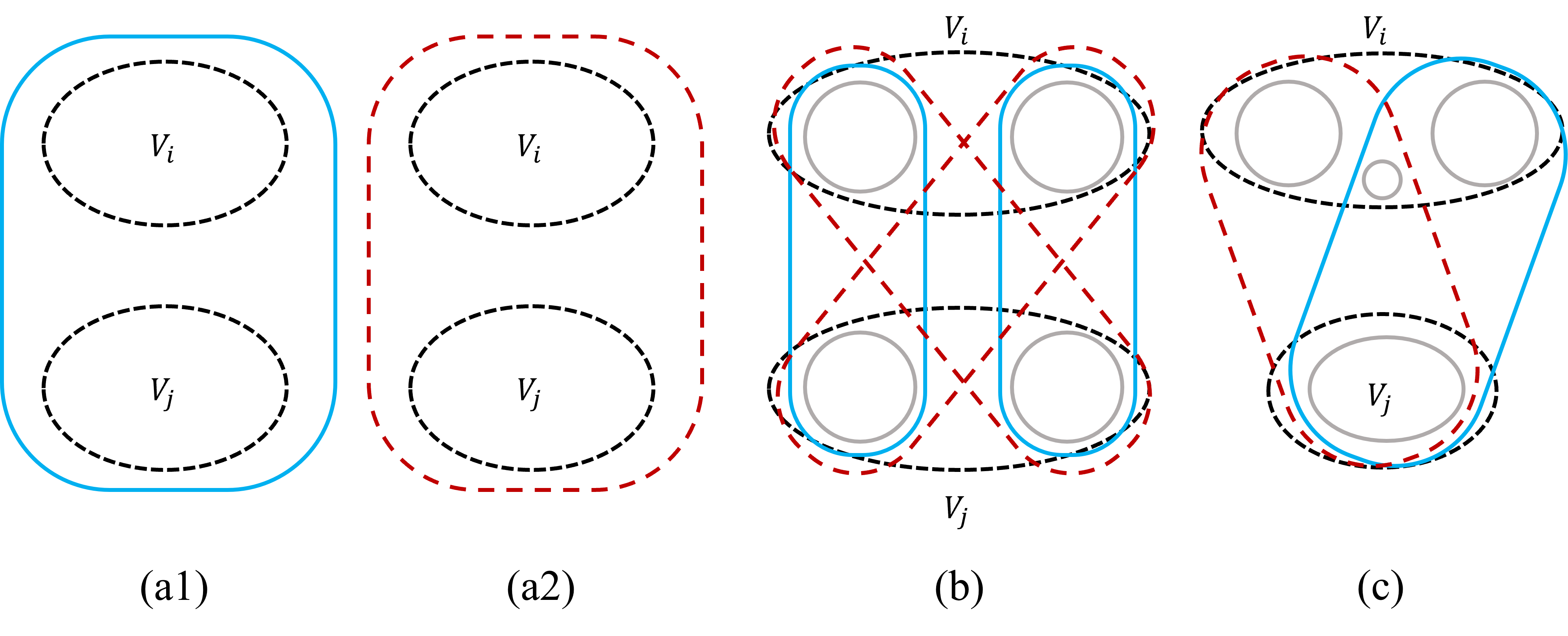}
    \caption{The cases between pairs in Lemma~\ref{lem:case3comps}}
    \label{fig:3complemma2}
\end{figure}

For any case constellation between the pairs $(V_i,V_j)$, we aim to exhibit at most two components other than $C_1,C_2,C_3$ covering $V_1\cup V_2\cup V_3$. Together with $C_1$, $C_2$, and $C_3$, this will yield a $2$-cover of the entire graph $K_n$ by at most $5$ components (since all vertices outside $V_1\cup V_2\cup V_3$ are in at least two of the $C_i$ and so are already $2$-covered). Applying Corollary~\ref{cor:fracbound} with $(\gamma,x)=(2, 5)$ would then give the claimed bound of
\[
z\geq \frac{2}{14+\sqrt{96}}.
\]

If case (a) (that is, either (a1) or (a2)) occurs for at least two pairs $(V_i,V_j)$, then the relevant components readily give a cover of $V_1\cup V_2\cup V_3$. Hence, from now on we can assume that this is not the case. 

Suppose case (a1) occurs on one of the pairs, say $(V_1,V_2)$. If $(V_1,V_3)$ satisfies case (b), then $V_1\cup V_2\cup V_3$ is covered by a single blue component. If $(V_1,V_3)$ satisfies case (c) (in either direction), then $V_1\cup V_2$ is covered by a blue component $B$ and $V_3\setminus B$ is covered by a non-blue component. Thus, we can assume from now on that case (a1) never happens.

If case (b) occurs, say on $(V_1,V_2)$, then one of the two remaining pairs, say $(V_1,V_3)$, must satisfy case (b) or (c). If $(V_1,V_3)$ satisfies case (b), then $V_1\cup V_2\cup V_3$ is covered by at most two blue components. The same is true if $(V_1,V_3)$ satisfies case (c) directed away from $V_1$. Lastly, if $(V_1,V_3)$ satisfies (c) directed toward $A_1$, then $V_1\cup V_2$ is covered by a blue component $B'$ and $V_3\setminus B'$ is covered by a non-blue component. Hence, we may also assume from now on that case (b) never happens. In other words, each pair satisfies either case (a2) or (c) and the former occurs at most once.  

If all three pairs $(V_i,V_j)$ satisfy case (c), then the respective `case (c) directions' result in an auxiliary $3$-vertex tournament, which is either cyclic or transitive. In the cyclic case we again have that $V_1\cup V_2\cup V_3$ is contained in a single blue component. In the transitive case, without loss of generality, let each $(V_i,V_j)$ for $i<j$ be directed toward $V_i$. Then again $V_1\cup V_2$ lies in a single blue component $B''$ and $V_3\setminus B''$ can be covered by a non-blue component.

Hence, we may assume that one pair, say $(V_1,V_2)$, is of type (a2), with $D$ being the non-blue component between them, and the other two pairs are of type (c). If either of these type (c) pairs is directed toward $V_3$, then $V_1\cup V_2\cup V_3$ is covered by $D$ and the corresponding blue component. Thus, the last case to consider is when both $(V_1,V_3)$ and $(V_2,V_3)$ are directed away from $V_3$. 

In that case, let $B_1$ and $D_1$ be the blue and non-blue components between $(V_1,V_3)$, respectively, and define $B_2$ and $D_2$ similarly with respect to $(V_2,V_3)$. If $B_1\cap B_2\cap V_3\neq \emptyset$, then $B_1$ and $B_2$ coalesce into a single component $B$ in $K_n$ and $V_3\setminus B$ can be covered by a non-blue component. On the other hand, if $B_1\cap B_2\cap V_3= \emptyset$, then 
$$D_1\cup D_2 \supseteq (V_3\setminus B_1)\cup (V_3\setminus B_2)= V_3$$
and, since $D_1$ and $D_2$ cover $V_1$ and $V_2$, respectively, we again obtain that
two components, namely $D_1$ and $D_2$, cover $V_1\cup V_2\cup V_3$. Thus, in every case, we have a component with at least $\frac{2}{14+\sqrt{96}} \binom{n}{2}$ edges, as desired.
\end{proof}

The lower bound of $\frac{2}{14+\sqrt{96}}$ in the conclusion of Lemma~\ref{lem:case3comps} can be improved slightly through a more careful analysis. However, we chose to omit this more involved proof, since the improved estimate is not needed for the tight case in Theorem~\ref{thm:4main}.

For any vertex $v$ and any component $C$ with $v\in C$, we can assume there is a vertex $w\in C$ not in any other component containing $v$; otherwise, we are done by Lemma~\ref{lem:case3comps}. In particular, we can assume that no component has its vertex set contained entirely within the vertex set of another component. 

For the sake of clarity, we say that two vertices $v_1$ and $v_2$ are \emph{joined} in a given color if the edge between them is of that color, while a set of vertices is \emph{connected} in a given color if they share a component of that color. 
Call two vertices $v_1$ and $v_2$ \emph{equivalent} if they are contained in a common component, that is, they are connected, in each color. Clearly, this is an equivalence relation, so we may speak of \emph{equivalence classes} with respect to it. Note that if $v_1$ and $v_2$ share exactly three components, then, considering a vertex $w$ outside of these three components (if there is no such $w$, we are again done by Lemma~\ref{lem:case3comps}), both $v_1$ and $v_2$ must be joined to $w$ via the fourth color, meaning that they share all four components, a contradiction. Hence, we can assume that any two vertices are either equivalent or share at most two components. If $v_1$ and $v_2$ share precisely two components, let us call them \emph{biconnected}.

The next two lemmas cover a pair of general cases where we again get a lower bound of at least $\frac{2}{14+\sqrt{96}} \binom{n}{2}$ edges.

\begin{lemma}\label{lem:disjoint}
If there is a pair of components of different colors that do not intersect, then there is a monochromatic component with at least $\frac{2}{14+\sqrt{96}} \binom{n}{2}$ edges.
\end{lemma}

\begin{proof}
By symmetry, we can assume that the two disjoint components $R$ and $B$ are red and blue, respectively. All edges between $R$ and $B$ are orange or yellow, so we can apply Lemma~\ref{lem:2colorbip} to the complete bipartite graph $G_{R,B}$ between them. If either $R$ or $B$ is contained entirely in a single orange or yellow component in $G_{R,B}$ and thus in $G$, we are done by Lemma~\ref{lem:case3comps}, so assume otherwise. Then only case~(b) of Lemma~\ref{lem:2colorbip} can apply, where there are exactly two orange components and two yellow components in $G_{R,B}$ (none of which can coalesce together in $G$). 
It is easy to see that in this case any two vertices in $R$ that share an orange component also share a yellow component and are therefore equivalent. Hence, there are exactly two equivalence classes of vertices in $R$ and, likewise, exactly two equivalence classes in $B$. Taking representatives $v_1,v_2\in R$ and $w_1,w_2\in B$ of these equivalence classes, let
\[
\mathcal{X}=\{C\in \mc:\: C\cap \{v_1,v_2,w_1,w_2\}\neq \emptyset \} \setminus \{R,B\}.
\]
Define $w(C)=1$ when $C\in \mathcal{X}$ and $w(C)=0$ otherwise. We claim that this choice of $w$ satisfies the conditions of Corollary~\ref{cor:fracbound} with $(\gamma,x)=(3,8)$. If $\{v_1,v_2\}\subseteq C$ for any component $C\neq R$, then $R\subseteq C$, a contradiction, so $v_1$ and $v_2$ do not share any non-red components; likewise, $w_1$ and $w_2$ do not share any non-blue components. Thus, $\mathcal{X}$ contains exactly two components of each color, so $\sum_{C\in \mc} w(C)= |X|=8$. Moreover, by construction, for $u\in R\cup B$, we have $\sum_{C\ni u} w(C)=3$.

Now consider a fixed $u\notin R\cup B$. Since the two equivalence classes in $R$ do not share any non-red components, there are at least two colors among the edges between $u$ and $R$ and, similarly, at least two colors among the edges between $u$ and $B$. If there are exactly two colors among the edges between $u$ and $R\cup B$, then these colors must be orange and yellow, a contradiction because $R\cup B$ cannot be covered by the union of an orange component and a yellow one. So there are at least three colors among the edges between $u$ and $R\cup B$, which means $\sum_{C\ni u} w(C)\geq 3$. Thus, we can apply Corollary~\ref{cor:fracbound} with $(\gamma,x)=(3,8)$, yielding a monochromatic component with at least $\frac{2}{17+\sqrt{33}} \binom{n}{2}>\frac{2}{14+\sqrt{96}}\binom{n}{2}$ edges, as desired.
\end{proof}

\begin{lemma}\label{lem:bic}
If there is a pair of biconnected vertices, then there is a monochromatic component with at least $\frac{2}{14+\sqrt{96}} \binom{n}{2}$ edges.
\end{lemma}

\begin{proof}
Suppose that $v_1$ and $v_2$ are biconnected. By symmetry, we can assume that the two components they share are red and orange and that there are at least as many orange components as red components in our coloring. We can also assume that there are at least two red components; otherwise, we are done by Lemma~\ref{lem:case3comps}. If there are exactly two orange components, then applying Proposition~\ref{prop:mainbound} to the set of red and orange components with $(\gamma,x)=(2,4)$ yields $z \geq \frac{2}{12+\sqrt{80}}$ and we are again done, so we can assume that there are at least three orange components. Let $R$ and $O$ be the joint red and orange components, respectively, of $v_1$ and $v_2$ and, for $i\in\{1,2\}$, let the yellow and blue components containing $v_i$ be $Y_i$ and $B_i$, respectively. Let $T=V\setminus R$. 

Since $\{R,O,Y_1,B_1\}$ are the four components of $v_1$, their union is the whole of $V$. Hence, $D:=B_1\setminus (R\cup O)\neq \emptyset$, as otherwise $\{R,O,Y_1\}$ would satisfy the assumptions of Lemma~\ref{lem:case3comps} and, analogously, $E:=Y_1\setminus (R\cup O)\neq \emptyset$. Observe that $v_2$ must be joined in yellow to all vertices in $D$ and in blue to all vertices in $E$. This means that $D\cap E=\emptyset$ and, consequently, $v_1$ is joined to all of $D$ in blue and to all of $E$ in yellow. Furthermore, $E=B_2\setminus (R\cup O)$ and $D=Y_2\setminus (R\cup O)$.

\begin{figure}[h]
    \centering
    \includegraphics[scale=0.6]{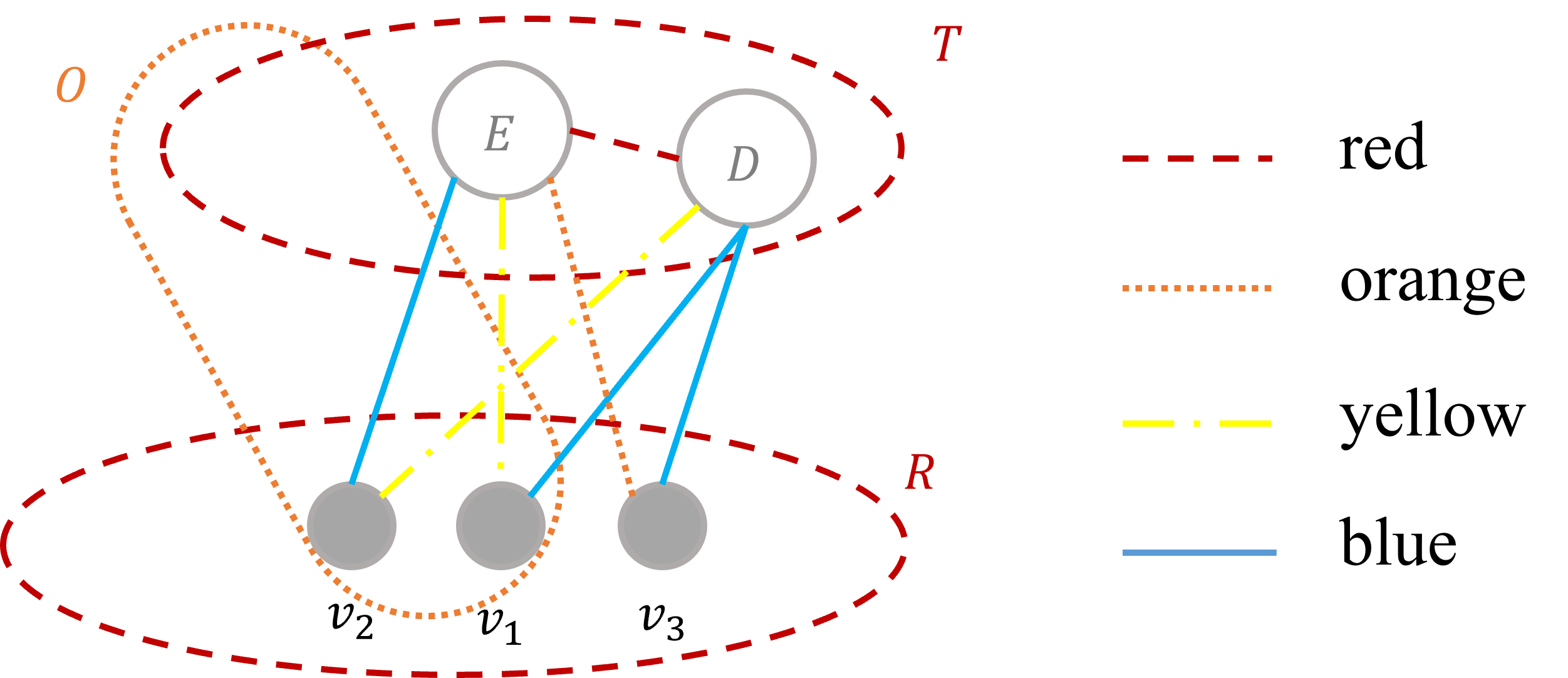}
    \caption{The situation in Lemma~\ref{lem:bic}}
    \label{fig:3complemma3}
\end{figure}

If there is a vertex $w\in Y_1\cap B_1\cap R\setminus O$, then $w$ shares exactly three components with $v_1$, a contradiction. Hence, we may conclude that 
\begin{equation}\label{eq:rbicon}
Y_1\cap B_1\setminus O=Y_1\cap B_1\setminus (R\cup O)=D\cap E=\emptyset
\end{equation}
and, similarly, $Y_2\cap B_2\setminus O=\emptyset$.

Now consider a vertex $v_3\in R\setminus O$. If $v_3$ is joined in orange to all of $T\setminus O=D\cup E$, then either there are only two orange components ($O$ and the component of $v_3$) or there is an orange component whose vertex set is contained within $R$; either case is a contradiction. Therefore, $v_3$ is joined in blue or yellow to some vertex in $T\setminus O$, which means it is in one of $B_1, B_2, Y_1, Y_2$. By symmetry, we can assume it is in $B_1$, so by~\eqref{eq:rbicon} we have $v_3\notin Y_1$. We conclude that $v_3$ must be joined to all of $E$ in orange (see Figure~\ref{fig:3complemma3}). Since $v_3$ was chosen from $R\setminus O$ arbitrarily, every vertex in $R\setminus O$ is joined either to all of $D$ or to all of $E$ in orange. If any orange component does not intersect $R$, we are done by Lemma~\ref{lem:disjoint}. Otherwise, we have at most three orange components in total, so by our earlier assumption it must be exactly three. In particular, there will be a vertex in $R\setminus O$ connected in orange to all of $D$ and none of $E$. It follows that every edge between $D$ and $E$ must be red, so all of $D \cup E=V\setminus (R\cup O)$ is contained in a single red component. If there is a third red component, it must lie entirely inside $O$, so we are done by Lemma~\ref{lem:case3comps}. Therefore, we can assume there are exactly two red components.

Next, note that $v_1$ and $v_3$ share components in red and blue but not in orange, which means $v_1$ and $v_3$ are biconnected. Therefore, by the same argument as above (with orange and blue swapped), the total number of blue components is also three. Let $B_3$ be the third blue component and note that $B_3\subseteq R\cup O$. Consequently, we must have $B_3\cap R\setminus O\neq \emptyset$, as otherwise $B_3$ would be contained in $O$. Let $v_4\in B_3\cap R\setminus O$ be an arbitrary vertex. By using similar reasoning as for $v_3$, we have that $v_4$ is in one of $B_1,B_2,Y_1,Y_2$; since $v_4\in B_3$, it must be in one of $Y_1$ and $Y_2$. We conclude, as before for $v_3$ and blue, that there are exactly three yellow components. 

Thus, we have exactly two red components and exactly three components of every other color, for a total of exactly $11$ monochromatic components. By the pigeonhole principle, one of these components must have at least $\frac{1}{11}\binom{n}{2}>\frac{2}{14+\sqrt{96}}\binom{n}{2}$ edges, as needed.
\end{proof}

We are now ready to complete the proof of our main theorem.

\begin{proof}[Proof of Theorem~\ref{thm:4main}]
If the hypotheses of any of Lemmas~\ref{lem:case3comps},~\ref{lem:disjoint}, or~\ref{lem:bic} hold, we have a component with at least $\frac{2}{14+\sqrt{96}}\binom{n}{2}$ edges, as needed, so we can assume otherwise. Thus, every pair of components of different colors intersects and any pair of vertices sharing at least two components are equivalent.

We claim that in this case there are at most three components of each color. Without loss of generality, let red be a color with the fewest components. Let $R$ be a red component and fix a vertex $v\notin R$. Let $O$, $Y$, and $B$ be the orange, yellow, and blue components of $v$, respectively. By Lemma~\ref{lem:bic}, we can assume that all vertices in $R\cap O$ are equivalent and similarly for $R\cap Y$ and $R\cap B$. However, $O\cup B\cup Y\supseteq R$, as every vertex in $R$ is adjacent to $v$ via a non-red edge. So the vertices of $R$ can be partitioned into at most three equivalence classes. Thus, there are at most three components of each color that intersect $V(R)$ and since, by Lemma~\ref{lem:disjoint}, we can assume that $R$ intersects every non-red component, there are at most three components of each color in total, as claimed.

If there is a color with fewer than three components, then $|\mc|\leq 11$ and there is some component with at least $\frac{1}{11}\binom{n}{2}$ edges. Otherwise, we will show that our coloring matches the extremal construction claimed. Indeed, by assumption, the equivalence class of a vertex is determined by its red and orange components. Let $R=R_1,R_2,R_3$ be the red components and $O_1,O_2,O_3$ the orange components and let $V_{ij}=R_i\cap O_j$ for $1\leq i,j\leq 3$. The $V_{ij}$ form a partition of $V(G)$ and each is an equivalence class of vertices. Since we can assume that every pair of components of different colors intersects, the $V_{ij}$ are all non-empty.

Every component contains exactly three equivalence classes: one for each orange component if it is red and one for each red component if it is not red. Each pair of the three vertex sets $V_{11},V_{22},V_{33}$ must share at least one component, each of which must be yellow or blue; one of the two colors occurs at least twice, say yellow. Then $V_{11}\cup V_{22}\cup V_{33}$ is connected in yellow and in fact must form a yellow component, since each yellow component contains exactly three equivalence classes. Then neither $V_{13}$ nor $V_{31}$ shares a red, orange, or yellow component with $V_{22}$, so $V_{13} \cup V_{22} \cup V_{31}$ forms a blue component and, similarly, so do $V_{12} \cup V_{21} \cup V_{33}$ and $V_{11}\cup V_{23}\cup V_{32}$. Repeating this argument shows that $V_{12}\cup V_{23}\cup V_{31}$ forms another yellow component, as does $V_{13}\cup V_{21}\cup V_{32}$. Hence, we are in exactly the extremal configuration claimed.

Thus, only the claimed extremal configuration can have fewer than $\frac{2}{14+\sqrt{96}}\binom{n}{2}$ edges in every component and in this extremal configuration, since there are $12$ components, we instead get a lower bound of $\frac{1}{12}\binom{n}{2}$ edges, as desired.
\end{proof}

As in \cite{conlon2021ramsey} and \cite{luo2021connected}, it is possible to amend our argument to show that every $4$-coloring of $K_n$ contains a monochromatic trail or circuit of length at least $(\frac{1}{24} + o(1))n^2$. We omit the details, but briefly note the main idea, which is to delete a sparse subgraph in order to guarantee that each component is Eulerian and then work around these omitted edges.

\vspace{3mm}

\bibliographystyle{acm}
\bibliography{main}

\begin{thebibliography}{10}

\bibitem{conlon2021ramsey}
{\sc Conlon, D., and Tyomkyn, M.}
\newblock Ramsey numbers of trails and circuits.
\newblock {\em Preprint available at arXiv:2109.02633 [math.CO]\/}.

\bibitem{debiasio2021ryser}
{\sc DeBiasio, L., Kamel, Y., McCourt, G., and Sheats, H.}
\newblock Generalizations and strengthenings of {R}yser's conjecture.
\newblock {\em Electron. J. Combin. 28}, 4 (2021), Paper No. 4.37.

\bibitem{debiasio2020bipartite}
{\sc DeBiasio, L., Krueger, R.~A., and S\'{a}rk\"{o}zy, G.~N.}
\newblock Large monochromatic components in multicolored bipartite graphs.
\newblock {\em J. Graph Theory 94}, 1 (2020), 117--130.

\bibitem{furedi81ryser}
{\sc F\"{u}redi, Z.}
\newblock Maximum degree and fractional matchings in uniform hypergraphs.
\newblock {\em Combinatorica 1}, 2 (1981), 155--162.

\bibitem{furedi89covering}
{\sc F\"{u}redi, Z.}
\newblock Covering the complete graph by partitions.
\newblock {\em Discrete Math. 75}, 1-3 (1989), 217--226.

\bibitem{gyarfas1977partition}
{\sc Gy{\'a}rf{\'a}s, A.}
\newblock Partition coverings and blocking sets in hypergraphs.
\newblock {\em Communications of the Computer and Automation Institute of the
  Hungarian Academy of Sciences 71\/} (1977), 62 pp.

\bibitem{Gyarfas2011Survey}
{\sc Gy\'{a}rf\'{a}s, A.}
\newblock Large monochromatic components in edge colorings of graphs: a survey.
\newblock In {\em Ramsey theory}, vol.~285 of {\em Progr. Math.}
  Birkh\"{a}user/Springer, New York, 2011, pp.~77--96.

\bibitem{gyarfas2017mindeg}
{\sc Gy\'{a}rf\'{a}s, A., and S\'{a}rk\"{o}zy, G.~N.}
\newblock Large monochromatic components in edge colored graphs with a minimum
  degree condition.
\newblock {\em Electron. J. Combin. 24}, 3 (2017), Paper No. 3.54.

\bibitem{luo2021connected}
{\sc Luo, S.}
\newblock On connected components with many edges.
\newblock {\em Preprint available at arXiv:2111.13342 [math.CO]\/}.

\bibitem{Ruszinko2012diameter}
{\sc Ruszink\'{o}, M.}
\newblock Large components in {$r$}-edge-colorings of {$K_n$} have diameter at
  most five.
\newblock {\em J. Graph Theory 69}, 3 (2012), 337--340.

\end{thebibliography}

\end{document}